\documentclass{amsart}

\usepackage{amsmath, amssymb, amsfonts, amscd, mathrsfs, color}

\usepackage{type1cm} 

\usepackage{graphicx}        % standard LaTeX graphics tool
\usepackage[bottom]{footmisc}% places footnotes at page bottom

\newcommand{\R}{\mathbb{R}} % reals
\newcommand{\C}{\mathbb{C}} % complex numbers
\newcommand{\Z}{\mathbb{Z}} % integers
\newcommand{\N}{\mathbb{N}} % naturals

\newcommand{\vG}{\varGamma}

\newcommand{\T}{\mathscr{T}}
\newcommand{\K}{\mathscr{K}}
\newcommand{\Hel}{\mathrm{H}}

\newcommand{\B}{\mathcal{B}}

\newcommand{\proj}{\mathbb{P}}
\newcommand{\Proj}{\mathcal{P}}

\newcommand{\wh}[1]{\widehat{#1}}

\newtheorem{theorem}{Theorem}[section]
\newtheorem{proposition}{Proposition}[section]
\newtheorem{lemma}{Lemma}[section]
\newtheorem{example}{Example}[section]
\newtheorem{remark}{Remark}[section]

\begin{document}

\title[Time-frequency data approximations]{Data approximation with time-frequency invariant systems}

\author[D. Barbieri]%, C. Cabrelli, E. Hern\'andez, U. Molter]%
{Davide Barbieri }%and Carlos Cabrelli and Eugenio Hern\'andez and Ursula Molter}
\address{Davide Barbieri and Eugenio Hern\'andez at Universidad Aut\'onoma de Madrid, Campus Cantoblanco, 28049 Madrid, Espa\~na}
\email {davide.barbieri@uam.es, eugenio.hernandez@uam.es}
\author[C. Cabrelli]{C. Cabrelli}
\address{Carlos Cabrelli and Ursula Molter at Departamento de Matem\'atica, Universidad de Buenos Aires, and Instituto de Matem\'atica ``Luis Santal\'o'' (IMAS-CONICET-UBA), 1428 Buenos Aires, Argentina}
\email{carlos.cabrelli@gmail.com, umolter@gmail.com}
\author[E. Hern\'andez]{E. Hern\'andez}
\author[U. Molter]{U. Molter}

\begin{abstract}
In this paper we prove the existence of a time-frequency space that best approximates a given finite set of  data. Here best approximation is in the least square sense, among all time-frequency spaces with no more than a prescribed number of generators. We provide a formula to construct the generators from the data and give the exact error of approximation. The setting is in the space of square integrable functions defined on a second countable LCA group and  we use  the Zak transform as the main tool.
\end{abstract}

\maketitle

\numberwithin{equation}{section}

\section{Introduction and main result}\label{Sec1}

Time-frequency systems, also called Gabor or Weyl-Heisenberg systems in the literature, are used extensively in the theory of communication, to analyze continuous signals, and to process digital data such as sampled audio or images. 

Time-frequency spaces try to represent features of both a function and its frequencies by decomposing the signal into time-frequency atoms given by modulations and translations of a finite number of functions \cite{Gro01}. If one looks at a musical score, on the horizontal axis the composer represents the time, and on the vertical axis the ``frequency'' given by the amplitude of the signal at that instant. Finding {\em sparse representations} (i.e. spaces generated by a small set of functions) will be useful for example in classification tasks.

In numerical applications to time-dependent phenomena, one often encounters uniformly sampled signals of finite length, i.e. vectors of $d$ elements, such as audio signals with a constant sampling frequency. In this case the most direct approach is to consider Fourier analysis on the cyclic group $\Z_d$.

To include a large variety of situations, our setting will be that of a locally compact abelian (LCA) group. The general construction developed in this paper will be specialised to the cyclic group $\Z_d$ in Example \ref{ex:finite}.

In this paper $G = (G,+)$ will be a second countable LCA group, that is, an abelian group endowed with a locally compact and second countable Hausdorff topology for which $(x,y) \mapsto x-y$ is continuous from $G\times G$ into $G$. We denote by $\widehat G$ the dual group of $G$, formed by the characters of $G$: an element $\alpha \in \widehat G$ is a continuous homomorphism from $G$ into $\mathbb T = \{z\in \C: |z|=1\}$. The action of $\alpha$ on $x\in G$ will be denoted by $(x,\alpha):= \alpha(x),$ to reflect the fact that the dual of $\widehat G$ is isomorphic to $G$, and therefore $x$ can also act on $\alpha$. For $\alpha_1, \alpha_2 \in \widehat G$ the group law is denoted by $\alpha_1\cdot \alpha_2$, so that $(x, \alpha_1\cdot \alpha_2) = (x, \alpha_1)(x, \alpha_2).$

A uniform lattice, $L\subset G$, is a subgroup of $G$ whose relative topology is the discrete one and for which $G/L$ is compact in the quotient topology. The annihilator of $L$ is $L^\perp = \{\alpha \in \widehat G : (\ell, \alpha)=1 \,\  \forall \ell \in L\}.$
Since $L^\perp \approx \widehat {(G/L)}$ (\cite{Rud92}, Theorem 2.1.2) and $G/L$ is compact, $L^\perp$ is discrete (\cite{Rud92}, Theorem 1.2.5). In particular, since $G$ is second countable, $\widehat{G}$ is also second countable, so both discrete groups $L$ and $L^\perp$ are countable.

Lel $L$ be a uniform lattice in the LCA group $G$ and $\B \subset L^\perp$ be a uniform lattice in the dual group $\widehat G.$ For $f\in L^2(G), \, \ell \in L$, and $\beta \in \B$ let $T_\ell f (x) = f(x-\ell), x\in G,$ be the translation operator, and $M_\beta f (x) = (x,\beta)f(x), x\in G$, be the modulation operator. The collection
$$
\{ T_\ell M_\beta f: \ell \in L, \beta \in \B \},
$$ 
is the time-frequency system generated by $f\in L^2(G).$

Since $\B \subset L^\perp$, we have $T_\ell M_\beta f = M_\beta T_\ell f$ for all  $f\in L^2(G),\, \ell \in L$, and $\beta \in \B$. Thus $\Pi(\ell, \beta):= T_\ell M_\beta$ is a unitary representation of the abelian group $\Gamma := L\times \B,$ with operation $(\ell_1, \beta_1)\cdot (\ell_2, \beta_2) = (\ell_1 + \ell_2, \beta_1 \cdot \beta_2)$, in $L^2(G).$

A closed subspace $V$ of $L^2(G)$ is said to be $\Gamma$-invariant (or time-frequency invariant) if for every $f\in V$, $\Pi(\ell, \beta)f \in V$ for every $(\ell, \beta)\in \Gamma.$ All $\Gamma$-invariant subspaces $V$ of $L^2(G)$ are of the form
$$
V=S_\Gamma(\mathcal A) := \overline{\mbox{span}\{ T_\ell M_\beta \varphi: \varphi \in \mathcal A, (\ell,\beta)\in \Gamma \}}^{L^2(G)}
$$
for some countable collection $\mathcal A$ of elements of $L^2(G).$ If $\mathcal A$ is a finite collection we say that $V=S_\Gamma(\mathcal A)$ has finite length, and $\mathcal A$ is a set of generators of $V$. We call the length of $V$, denoted length($V$), the minimum positive integer $n$ such that $V$ has a set of generators with $n$ elements.

We now state our approximation problem. Let $\mathcal F = \{ f_1, f_2, ..., f_m\} \subset L^2(G)$ be a set of functional data. Given a closed subspace $V$ of $L^2(G)$ define
\begin{equation} \label{Eq1-1}
  \mathcal E(\mathcal F; V):= \sum_{j =1}^m \| f_j - \proj_V f_j\|_{L^2(G)}^2\,
\end{equation}
as the error of approximation of $\mathcal F$ by $V$, where $\proj_V$ denotes the orthogonal projection of $L^2(G)$ onto $V.$

Is it possible to find a $\Gamma$-invariant space of length at most $n < m$ that {\em best} approximates our functions, in the sense that 
$$
\mathcal E(\mathcal F; S_\Gamma\{ \psi_1, ... , \psi_n \}) \leq \mathcal E(\mathcal F; V)
$$
for all $\Gamma$-invariant subspaces $V$ of $L^2(G)$ with length($V$) $\leq n$?

This question is relevant in applications. For example, if $\{f_1, \dots, f_m\}$ are audio signals, {\em the best} $\Gamma$-invariant space  provides a time-frequency {\em optimal} model to represent these signals.

The answer to this question is affirmative, and is given by the main theorem of this work.
\begin{theorem}\label{Thm:main}
Let $G$ be a second countable LCA group, $L$ and $\B$ be uniform lattices in $G$ and $\widehat G$ respectively, with $\B\subset L^\perp$. For each set of functional data $\mathcal F = \{ f_1, f_2, ..., f_m\} \subset L^2(G)$ and each $n\in \N, \, n < m,$ there exists $\{ \psi_1, ... , \psi_n  \} \subset L^2(G)$ such that
$$
\mathcal E(\mathcal F; S_\Gamma\{ \psi_1, ... , \psi_n \}) \leq \mathcal E(\mathcal F; V)
$$
for all $\Gamma$-invariant subspaces $V$ of $L^2(G)$ with length($V$) $\leq n$.
\end{theorem}

\begin{remark}

Observe that, in the previous statement, some of the generators $\{ \psi_1, ... , \psi_n  \}$ may be zero. In this case, the length of $S_\Gamma\{ \psi_1, ... , \psi_n \}$ would be strictly smaller than $n$.

\end{remark}

The proof of Theorem \ref{Thm:main} will follow the ideas originally developed in \cite{ACHM2007} for approximating data in $L^2(\R^d)$ by shift-invariant subspaces of finite length, and which have also been used in \cite{CMP2017, BCHM2019}.

We reduce the problem of finding the collection $\{ \psi_1, ... , \psi_n  \}$, whose existence is asserted in Theorem \ref{Thm:main}, to solve infinitely many approximation problems for data in a particular Hilbert space of sequences. This is accomplished with the help of an isometric isomorphism $H_\Gamma$ that intertwines the unitary representation $\Pi$ with the characters of $\Gamma$. This isometry $H_\Gamma$ generalizes the fiberization map of \cite{BDR1994} used in \cite{ACHM2007}, and has the properties of a Helson map as defined in \cite{BHP2018}(Definition 7). The definition and properties of $H_\Gamma$ are given in Section \ref{Sec2}.

The reduced problems are then solved by using Eckart-Young theorem as stated and proved in \cite{ACHM2007} (Theorem 4.1). The solutions of all of these reduced problems are patched together to finally obtain the proof of Theorem \ref{Thm:main} in Section \ref{Sec3}.

\section{An isometric isomorphism} \label{Sec2}

Let $G$ be a second countable LCA group, $L$ a uniform lattice in $G$, and $\B \subset L^\perp$ a uniform lattice in $\widehat G$ (see definitions in Section \ref{Sec1}). With $\Gamma = L\times \B$, each $\Gamma$-invariant subspace $V$ of $L^2(G)$ is of the form
$$
V=S_\Gamma(\mathcal A) := \overline{\mbox{span}\{ T_\ell M_\beta \varphi: \varphi \in \mathcal A, (\ell,\beta)\in \Gamma \}}^{L^2(G)}
$$
for some countable set $\mathcal A \subset L^2(G)$. Therefore 
$$
V=S_L(\{ M_\beta \varphi: \varphi  \in \mathcal A, \beta \in \B \}) 
$$
is also an $L$-invariant subspace, that is $T_\ell f \in V$ for all $\ell \in L$ whenever $f\in V$. The theory of shift-invariant spaces on LCA groups, as developed in \cite{CP2010}, can be applied to this situation.

Let $T_{L^\perp} \subset \widehat{G}$ be a measurable cross-section of $\widehat{G}/L^\perp$. The set $T_{L^\perp}$ is in one to one correspondence with the elements of $\widehat{G}/L^\perp$, and $\{ T_{L^\perp} + \lambda : \lambda \in L^\perp  \}$ is a tiling of $\widehat G$. 

Let $\widehat f (\omega) := \int_G f(x) \overline{(x, w)} dx$ denote the unitary Fourier transform of $f\in L^2(G)\cap L^1(G)$ and extended to $L^2(G)$ by density. By Proposition 3.3 in \cite{CP2010} the mapping $\T : L^2(G) \to L^2(T_{L^\perp},\ell^2(L^\perp))$ given by
\begin{equation} \label{Eq2-1}
\T f(\omega) = \{\wh{f}(\omega + \lambda)\}_{\lambda \in L^\perp} , \ f \in L^2(G),
\end{equation}
is an isometric isomorphism. Moreover, since $V\subset L^2(G)$ is an $L$-invariant space, it has an associated measurable range function 
$$
J : T_{L^\perp} \longrightarrow \{ \mbox{closed subspaces of} \ \ell^2(L^\perp) \}
$$
such that (See Theorem 3.10 in \cite{CP2010})
\begin{equation} \label{Eq2-2}
J(\omega) = \overline{ \mbox{span} \, \{ \T(M_\beta \varphi)(\omega) : \beta \in \B, \, \varphi \in \mathcal A \}  }^{\ell^2(L^\perp)} \,, \ \mbox{a.e}\ \omega \in T_{L^\perp}.
\end{equation}

Using the definition of $\T$ given in (\ref{Eq2-1}), for each $\beta \in \B$ and each $\varphi \in L^2(G)$ we have
\begin{equation} \label{Eq2-3}
   \T(M_\beta \varphi)(\omega) = \{ \widehat {M_\beta \varphi}\,(\omega + \lambda) \}_{\lambda \in L^\perp} = \{ \widehat \varphi \,(\omega + \lambda - \beta) \}_{\lambda \in L^\perp} = t_\beta(\T \varphi (\omega))
\end{equation}
where $t_\beta : \ell^2(L^\perp) \longrightarrow \ell^2(L^\perp)$ is the translation of sequences in $\ell^2(L^\perp)$ by elements of $\beta \in \B,$ that is $t_\beta (\{ a(\lambda) \}_{\lambda \in L^\perp}) = \{ a(\lambda - \beta) \}_{\lambda \in L^\perp}.$ Therefore, $\T$ intertwines the modulations $\{M_\beta\}_{\beta \in \B}$ with the translations by $\B$ on $\ell^2(L^\perp).$

By equations \eqref{Eq2-2} and \eqref{Eq2-3}, for a. e. $\omega\in T_{L^\perp}$,
$$
J(\omega) = \overline{ \mbox{span} \, \{ t_\beta (\T \varphi(\omega)) : \beta \in \B, \, \varphi \in \mathcal A \}  }^{\ell^2(L^\perp)}.
$$
Therefore, $J(\omega)$ is a $\B$-invariant subspace of $L^2(L^\perp)$. We can apply the theory of shift-invariant spaces as developed in \cite{CP2010} to the discrete LCA group $L^\perp$ and its uniform lattice $\B$.

Let $\B^\perp$ be the annihilator of $\B$ in the compact group $\widehat {L^\perp} \subset G$, that is
\begin{equation}  \label{Eq2-4}
\B^\perp = \{ b \in \widehat{L^\perp} : (b, \beta) = 1 \ \forall \beta\in \B   \}.
\end{equation}
Observe that $\B^\perp$ is finite, because it is a discrete subgroup of a compact group.

Let $T_{ \B^\perp} \subset \widehat{L^\perp}$ be a measurable cross-section of $ \widehat{L^\perp}/\B^\perp$.
The set $T_{ \B^\perp}$ is in one to one correspondence with the elements of $ \widehat{L^\perp}/ \B^\perp$ and $\{ T_{\B^\perp} + b : b \in \B^\perp  \}$ is a tiling of $\widehat{L^\perp}$. 

\begin{example}.
\emph{
Let $G = \R, L = \Z$ and $\B = n\Z \subset L^\perp = \Z \subset \widehat \R.$ Since $\widehat {L^\perp} = \widehat \Z \approx [0,1)$,\ \  $\ell \in \B^\perp$ if and only if $\ell \in [0,1)$ and $e^{2\pi i \ell \cdot nk}= 1$ for all $k\in \Z.$ Hence 
$$
\B^\perp = \{ 0, \frac1n, \dots, \frac{n-1}{n} \}.
$$
We can take $\displaystyle T_{\B^\perp} = [0, \frac1n)$. Notice that as a subgroup of $\widehat \R$ the annihilator of $\B$ is $\displaystyle \frac1n \Z.$
}
\end{example}

\begin{example}\label{ex:finite}
\emph{
Let $p, q \in \N$, $d = pq$, and $G = \Z_d = \{0, 1, \dots, d-1\}$. Let $L = \{0, p, 2p, \dots p(q-1)\} = \{n p : n=0,\dots,q-1\}\approx \Z_q$. Its annihilator lattice is
\begin{align*}
L^\perp & = \Big\{\lambda \in \{0, 1, \dots, d-1\} : e^{2\pi i \frac{\lambda n p}{d}} = 1 \ \forall \ n=0,\dots,q-1\Big\}\\
& = \{0, q, 2q, \dots q(p-1)\} = \{kq : k=0,\dots,p-1\}\approx \Z_p.
\end{align*}
A fundamental set $T_{L^\perp}$ for $L^\perp$ in $\wh{G} \approx \Z_d$  is
$
T_{L^\perp} = \{0,\dots,q-1\} \approx \Z_q.
$
The characters $\omega \in \wh{L^\perp} = \{$homomorphisms $: L^\perp \to \mathbb{T}\}$ of this group are of the form (see e.g. \cite{Deit2005} Lemma 5.1.3)
$
\omega_\nu(\lambda) = e^{2\pi i \frac{\lambda \nu}{p}}, \ \lambda \in L^\perp
$
for $\nu \in \{\frac{\ell}{q} : \ell = 0,\dots,p-1\} \approx \Z_p$.
}
{\emph
Suppose now that $p = rs$ for some $r,s \in \N$, and let $\B \subset L^\perp$ be
$$
\B = \{0, rq, 2rq, \dots, (s-1)rq\} = \{j r q : j=0,\dots,s-1\} \approx \Z_s.
$$
The annihilator of $\B$ in $\wh{L^\perp}$ thus reads
\begin{align*}
\B^\perp & = \Big\{b \in \{\frac{\ell}{q} : \ell = 0,\dots,p-1\} : e^{2\pi i \frac{b j r q}{p}} = 1 \ \forall j=0,\dots,s-1\Big\}\\
& = \{0,\frac{s}{q},\frac{2s}{q},\dots,\frac{s(r-1)}{q}\} = \{h \frac{s}{q} : h=0,\dots,r-1\}\approx \Z_r.
\end{align*}
A fundamental set in $\wh{L^\perp} = \{\frac{\ell}{q} : \ell = 0,\dots,p-1\}$ for $\B^\perp$ is \newline
$
T_{\B^\perp} = \Big\{0, \frac{1}{q}, \dots,\frac{s-1}{q}\Big\} \approx \Z_s.
$
}
\end{example}

By Proposition 3.3 in \cite{CP2010}, the mapping $\K : \ell^2(L^\perp) \to L^2(T_{ \B^\perp},\ell^2( \B^\perp))$ given by
\begin{eqnarray} 
\K (\{a(\lambda)\}_{\lambda\in L^\perp})(t) &=&
\{  (\{ a(\lambda)\}_{\lambda\in L^\perp})^\wedge\  (t+b) \}_{b \in \B^\perp} \nonumber \\
&=& \left\{\sum_{\lambda \in L^\perp} a(\lambda) \overline{(t+b, \lambda)}\right\}_{b \in  \B^\perp},  \label{Eq2-5}
\end{eqnarray}
is an isometric isomorphism. Moreover, each $\B$-invariant subspace $J(\omega), \ \omega \in T_{L^\perp},$ has an associated measurable range function
$$
J(\omega, \cdot): T_{ \B^\perp} \longrightarrow \{ \mbox{closed subspaces of} \ \ell^2(\B^\perp) \},
$$
such that for almost every $t\in T_{\B^\perp}$,
${\displaystyle
J(\omega, t) = \overline{ \mbox{span} \, \{ \K (\T \varphi)(\omega)) (t) : \varphi \in \mathcal A \}  }^{\ell^2(\B^\perp)} \,. }
$
From the definition of $\T$ given in \eqref{Eq2-1} and the definition of $\K$ given in \eqref{Eq2-5} we obtain
\begin{equation} \label{Eq2-6}
\K (\T \varphi)(\omega))(t) = \left\{\sum_{\lambda \in L^\perp} \widehat f(\omega + \lambda) \overline{(t+b, \lambda)} \right\}_{b \in  \B^\perp},
\end{equation}
when $f\in L^2(G), \, \omega \in T_{L^\perp},$ and $ t\in T_{\B^\perp}.$

For $f\in L^2(G), \, \omega \in \widehat G,$ and $ t\in G$ define
\begin{equation} \label{Eq2-7}
   \mathcal Z f (\omega, t) := \sum_{\lambda \in L^\perp} \widehat f(\omega + \lambda) \overline{(t, \lambda)}\,,
\end{equation}
the Zak transform of $\widehat f$ with respect to the lattice $L^\perp.$ Observe that in terms of this map, $\K (\T \varphi)(\omega))(t) = \{ \mathcal Z f (\omega, t+b) \}_{b \in \B^\perp}.$

To simplify the statement of the next theorem we write $X_\beta$ for the character on $G$ associated to $\beta\in \B$, that is $X_\beta : G \longrightarrow \mathbb T $ with $X_\beta(x)= (x,\beta)$ for all $x\in G$. Similarly $X_\ell$ will denote the character on $\widehat G$ associated to $\ell \in L$, that is $X_\ell : \widehat G \longrightarrow \mathbb T $ with $X_\ell(\omega)= (\ell,\omega)$ for all $\omega\in \widehat G$.

\begin{theorem}\label{Th2-1}
Let $G$ be a second countable  LCA group, $L$ and $\B$ be uniform lattices in $G$ and $\widehat G$ repectively, with $\B\subset L^\perp$. Let $\Gamma = L \times \B$ and for 
$f\in L^2(G), \,\omega \in T_{L^\perp},$ and $ t\in T_{\B^\perp}$ define
\begin{equation} \label{EqH}
H_\Gamma f (\omega, t) = \{ \mathcal Z f (\omega, t+b) \}_{b \in \B^\perp}.
\end{equation}
Then

1) The map $H_\Gamma$ intertwines $\Pi$ with the characters of $\Gamma$, that is
$H_\Gamma \Pi(\ell, \beta) f = X_{-\ell} X_{-\beta} H_\Gamma f$ for all  $f\in L^2(G), \ell\in L, \beta\in \B.$

2) The map $H_\Gamma$ defined in \eqref{EqH} is an isometric isomorphism from $L^2(G)$ onto $L^2(T_{L^\perp}\times T_{\B^\perp}, \ell^2(\B^\perp)).$
\end{theorem}

\begin{proof}
For each $b \in \B^\perp$, the definition of $\mathcal Z$ given in \eqref{Eq2-7} and the properties of the Fourier transform give
\begin{align*}
\mathcal Z \Pi(\ell, \beta) f (\omega,t+b) &  = \sum_{\lambda \in \Lambda^\perp} \wh{T_\ell M_\beta f}(\omega + \lambda) \overline{(t+b, \lambda)}\\
& = \sum_{\lambda \in \Lambda^\perp} \overline{ (\ell, \omega + \lambda)} \wh{ f}(\omega + \lambda - \beta) \overline{(t+b, \lambda)}\,.
\end{align*}
Using that $(\ell, \lambda)=1$ and the change of variables $\lambda - \beta = \lambda' \in L^\perp$ yields
$$
\mathcal Z \Pi(\ell, \beta) f (\omega,t+b) = 
\overline{ (\ell, \omega)} \sum_{\lambda' \in \Lambda^\perp} \wh{ f}(\omega + \lambda') \overline{(t+b, \lambda' + \beta)}\,.
$$
Using that $(t+b, \beta) = (t, \beta)\cdot (b,\beta) = (t,\beta)$ we obtain
\begin{align*}
\mathcal Z \Pi(\ell, \beta) f (\omega,t+b) &  = 
\overline{ (\ell, \omega)}\ \overline{(t,\beta)} \sum_{\lambda' \in \Lambda^\perp} \wh{ f}(\omega + \lambda') \overline{(t+b, \lambda')}\\
& = X_{-\ell}(\omega) X_{-\beta}( t) \mathcal Z f (\omega, t + b)\,.
\end{align*}
This proves $1)$. To prove $2)$ observe that by the definition of $H_\Gamma$ given in \eqref{EqH} together with \eqref{Eq2-6} and \eqref{Eq2-7} we have
$$
H_\Gamma f (\omega, t) = \K (\T  f(\omega))(t)\,.
$$

That $H_\Gamma$ is an isometry now follows from the fact that $\T$ and $\K$ are isometries in their respective spaces.

We need to prove that $H_\Gamma$ is onto. Since $\K : \ell^2(L^\perp) \to L^2(T_{ \B^\perp},\ell^2( \B^\perp))$ is an isometric isomorphism between Hilbert spaces, by Lemma \ref{Lemma} in the Appendix, the map 
$$
Q_\K : L^2(T_{L^\perp}, \ell^2(L^\perp)) \longrightarrow L^2(T_{L^\perp}, L^2(T_{\mathcal B^\perp}, \ell^2(\mathcal B^\perp))
$$
given by 
$$
(Q_\K f)(\omega) = \K(f(\omega)) , \ f \in L^2(T_{L^\perp}, \ell^2(L^\perp))
$$ 
is an isometric isomorphism. Moreover, by Fubini's theorem, the Hilbert spaces $L^2(T_{L^\perp}, L^2(T_{\mathcal B^\perp}, \ell^2(\mathcal B^\perp))$ and $L^2(T_{L^\perp}\times T_{\mathcal B^\perp}, L^2(\ell^2(\mathcal B^\perp))$ are also isomorphic and the isomorphism is given by $\Phi(f)(\omega, t) = f(\omega)(t)$, for $f \in L^2(T_{L^\perp}, L^2(T_{\mathcal B^\perp}, \ell^2(\mathcal B^\perp))$.

Let now $F\in L^2(T_{L^\perp}\times T_{\mathcal B^\perp}, L^2(\ell^2(\mathcal B^\perp))$. Choose $g\in L^2(T_{L^\perp}, \ell^2(L^\perp))$ such that $\Phi\circ Q_\K (g) = F.$ Hence
$$
F(\omega,t) = \Phi\circ Q_\K (g)(\omega, t) = Q_\K(g)(\omega)(t) = \K (g(\omega))(t).
$$
Choose now $f\in L^2(G)$ such that $\T(f) = g.$ Then 
$$
H_\Gamma f (\omega, t) = \K(\T f (\omega))(t) = F(\omega, t).
$$
This finishes the proof of the theorem.
\end{proof}

\begin{example}
\emph{For the cyclic group of Example \ref{ex:finite}, recall that, for} $f \in \C^d$
$$
\wh{f}(\omega) = \frac{1}{\sqrt{d}}\sum_{g=0}^{d-1} f(g) e^{-2\pi i \frac{g\omega}{d}} , \ \omega \in \{0,\dots,d-1\}.
$$
\emph{For} $t \in T_{\B^\perp} = \Big\{0, \frac{1}{q}, \dots,\frac{s-1}{q}\Big\}$\emph{, the Zak transform (\ref{Eq2-7}) thus reads}
\begin{align*}
\mathcal Z & f (\omega,t) = \sum_{k = 0}^{p-1} \wh{f}(\omega + kq) e^{-2 \pi i \frac{kqt}{p}}
= \sum_{k = 0}^{p-1} \frac{1}{\sqrt{d}}\sum_{g=0}^{d-1} f(g) e^{-2\pi i \frac{g(\omega + kq)}{d}} e^{-2 \pi i \frac{kqt}{p}}\\
& = \frac{1}{\sqrt{d}}\sum_{g=0}^{d-1} f(g) e^{-2\pi i \frac{g\omega}{d}} K(g + qt) = \frac{e^{2\pi i \frac{q t \omega}{d}}}{\sqrt{d}}\sum_{g=0}^{d-1} f(g - qt) e^{-2\pi i \frac{g\omega}{d}} K(g)
\end{align*}
\emph{where} $\displaystyle K(g) = \sum_{k = 0}^{p-1} \left(e^{-2 \pi i \frac{g}{p}}\right)^k = \left\{
\begin{array}{lll}
p &\text{if} & \ g \in L\\
0 &\text{if} & \ g \notin L
\end{array}
\right.
$\emph{. This gives}
$$
\mathcal Z f (\omega,t) = \sqrt{p} e^{2\pi i \frac{q t \omega}{d}}\,\frac{1}{\sqrt{q}}\sum_{n=0}^{q-1} f(pn - qt) e^{-2\pi i \frac{pn\omega}{q}}.
$$
% Note that, since $qt \in \Z$ for all $t \in T_{\B^\perp}$, this is a multiple of the $\Z_q$-Fourier transform (DFT) of a subsequence of $f$. Indeed, if, for $x \in \Z$, we define $f_x = \{f(pn - x) : n \in \{0,\dots,q-1\}\} \in \C^q$ and we denote by $\wh{f_x}$ its $\Z_q$ Fourier transform, we have
% $$
% \mathcal Z f (\omega,t) = \sqrt{p} e^{2\pi i \frac{q t \omega}{d}} \wh{f_{qt}}(\omega).
% $$
% Observe also that, for the construction of the Helson map (\ref{EqH}), $\omega$ is considered only in $T_{L^\perp} = \{0,\dots,q-1\}$. It reads explicitly
% \begin{align*}
% H_\Gamma f(\omega,t) & = \left\{ \frac{1}{\sqrt{d}}\sum_{g=0}^d f(g) e^{-2\pi i \frac{g\omega}{d}} K(g + qt + h s) \right\}_{h=0}^{r-1}\\
% & = \left\{ \sqrt{p} e^{2\pi i \frac{(q t + h s) \omega}{d}} \wh{f_{qt + hs}}(\omega)  \right\}_{h=0}^{r-1}
% \end{align*}
\end{example}

Before embarking in the proof of Theorem \ref{Thm:main}, which will be accomplished in Section \ref{Sec3}, we need an additional result.

Let $V=S_\Gamma(\mathcal A)$ be a $\Gamma$-invariant subspace of $L^2(G)$, where $\mathcal A \subset L^2(G).$ For each $(\omega, t)\in T_{L^\perp}\times T_{\B^\perp},$ consider the range function 
$$
J_V : \T_{L^\perp}\times T_{\B^\perp} \longrightarrow \{\mbox{closed subspaces of} \ \ell^2(\B^\perp) \}
$$ 
given by
\begin{equation} \label{EqJ}
J_V(\omega, t) : = \overline{ \mbox{span} \, \{ H_\Gamma \varphi (\omega,t) : \varphi \in \mathcal A \}  }^{\ell^2(\B^\perp)} \,.
\end{equation}

\begin{proposition} \label{Pro1}
With $V=S_\Gamma(\mathcal A)$ as above, let $\Proj_{J_V(\omega,t)}$ be the orthogonal projection of $\ell^2(\B^\perp)$ onto $J_V(\omega,t)$. Then, for all $f \in L^2(G)$ and $(\omega, t)\in T_{L^\perp}\times T_{\B^\perp},$
\begin{equation*}
\Hel_\vG \proj_{S_\Gamma(\mathcal A)} f (\omega,t) = \Proj_{J_V(\omega,t)} (H_\Gamma f(\omega,t)) \,.
\end{equation*}
\end{proposition}

\begin{proof}
Observe first that, since $H_\Gamma$ is an isometric isomorphism between Hilbert spaces, then
\begin{equation} \label{Eq2-10}
H_\Gamma \proj_{S_\Gamma(\mathcal A)} = \proj_{H_\Gamma (S_\Gamma(\mathcal A))}H_\Gamma .
\end{equation}
The set $\mathcal D := \{ X_\ell X_\beta : (\ell, \beta) \in \Gamma \}$ of characters of $\Gamma$ is a determining set for $L^1(T_{L^\perp}\times T_{\B^\perp})$ in the sense of Definition 2.2 in \cite{BR2015}, because
$$
\int_{T_{L^\perp}\times T_{\B^\perp}} f(\omega,t)X_\ell(\omega) X_\beta(t) d\omega dt = 0 \Rightarrow f = 0 \quad \forall \ f \in L^1(T_{L^\perp}\times T_{\B^\perp}).
$$
% {\color{violet} como no estoy muy segura si es OBVIO, preferiria poner aqui la definicion de "determining set" - salvo que para todos es totalmente claro....}  (\textcolor{blue}{EH: Queda engorroso poner la definici\'on en general y el lector interesado puede verlo en la referencia citada}), 
Indeed, this is Fourier uniqueness theorem since $T_{L^\perp}$ and $T_{\B^\perp}$ are relatively compact.

By $1)$ of Theorem \ref{Th2-1}, for all $f\in L^2(G)$, $H_\Gamma (T_\ell M_\beta f) = X_{-\ell} X_{-\beta} (H_\Gamma f)$. Thus, $H_\Gamma (S_\Gamma(\mathcal A))$ is $\mathcal D$-multiplicative invariant in the sense of Definition 2.3 in \cite{BR2015}. Indeed, if $ X_\ell X_\beta \in \mathcal D$, $F\in H_\Gamma (S_\Gamma(\mathcal A))$ writing $H_\Gamma f = F$ we have
$$
X_\ell X_\beta F = X_{\ell} X_{\beta} (H_\Gamma f) = H_\Gamma (T_{-\ell} M_{-\beta} f) \in H_\Gamma (S_\Gamma(\mathcal A))\,.
$$ 
By Theorem 2.4 in \cite{BR2015}, $J_V$ is a measurable range function. By Proposition 2.2 in \cite{BR2015}, 
$$
\proj_{H_\Gamma (S_\Gamma(\mathcal A))}(H_\Gamma f)(w,t) = \Proj_{J_V(\omega,t)} (H_\Gamma f(\omega,t)) \,.
$$
The result now follows from \eqref{Eq2-10}.
\end{proof}

\section{Solution to the approximation problem} \label{Sec3}

This section is dedicated to the proof of Theorem \ref{Thm:main}. Let $\mathcal F = \{f_1, \dots, f_m\} \subset L^2(G)$ be a collection of functional data. With the notation of Theorem \ref{Thm:main}, for each $n  <  m$ we need to find $\{\psi_1, \dots, \psi_n\} \subset L^2(G)$ such that $\mathcal E(\mathcal F; S_\Gamma\{\psi_1, \dots, \psi_n\})\leq \mathcal E(\mathcal F; V)$ for any $\Gamma$-invariant subspace $V$ of $L^2(G)$ of length less than or equal $n$. The definition of $\mathcal E(\mathcal F; V)$ is given in \eqref{Eq1-1} and for convenience of the reader we recall it here.
\begin{equation*} %\label{Eq1-1}
  \mathcal E(\mathcal F; V):= \sum_{j =1}^m \| f_j - \proj_V f_j\|_{L^2(G)}^2\ .
\end{equation*}

For a.e. $(\omega, t)\in T_{L^\perp}\times T_{\B^\perp}$ consider
$$
H_\Gamma(\mathcal F)(w,t) := \{ H_\Gamma f_1(\omega,t) , \dots, H_\Gamma f_m(\omega,t) \}\,.
$$
Let $G_{\mathcal F, \Gamma}(w,t)$ be the $m\times m$ $\C$-valued matrix whose $(i,j)$ entry is given by
$$
[G_{\mathcal F, \Gamma}(w,t)]_{i,j} = \langle H_\Gamma f_i(\omega,t) , H_\Gamma f_j(\omega,t) \rangle_{\ell^2(\B^\perp)}\,.
$$
The matrix $G_{\mathcal F, \Gamma}(w,t)$ is hermitian and its entries are measurable functions defined on $T_{L^\perp}\times T_{\B^\perp}$. Write
$$
\lambda_1(\omega, t) \geq \lambda_2(\omega, t) \geq \dots, \geq \lambda_m(\omega, t) \geq 0
$$
for the eigenvalues of $G_{\mathcal F, \Gamma}(w,t)$. By Lemma 2.3.5 in \cite{RS95} the eigenvalues $\lambda_i(\omega, t)$, $i=1, \dots, m$, are measurable and there exist corresponding measurable vectors $y_i(\omega, t) = (y_{i,1}(\omega,t), \dots, y_{i,m}(\omega,t))$ that are orthonormal left eigenvectors of the matrix $G_{\mathcal F, \Gamma}(w,t)$. That is,
\begin{equation} \label{Eq3-1}
  y_i(\omega,t)\, G_{\mathcal F, \Gamma}(w,t) = \lambda_i(\omega, t)\, y_i(\omega,t), \quad i=1, \dots, m.
\end{equation}

For $n\leq m$, define $q_1(\omega,t), \dots, q_n(\omega,t) \in \ell^2(\B^\perp)$ by
\begin{equation}  \label{Eq3-2}
    q_i(\omega,t) = \widetilde \sigma_i(\omega, t) \sum_{j=1}^m y_{i,j}(\omega, t)\, H_\Gamma f_j(\omega,t) \quad i=1, \dots, n,
\end{equation}
where 
$$\widetilde \sigma_i(\omega, t) = \begin{cases}
\frac{1}{\sqrt{\lambda_i(\omega,t)}} & \text{ if} \quad  \lambda_i(\omega,t) \neq 0\\
\qquad   0& \quad \text{otherwise.}
\end{cases}$$

By the Eckart-Young Theorem (see the version stated and proved in Theorem 4.1 of \cite{ACHM2007}), $\{q_1(\omega,t), \dots, q_n(\omega,t)\}$ is a Parseval frame for the space it generates $Q(\omega, t) := \mbox{span}\, \{q_1(\omega,t), \dots, q_n(\omega,t)\}$ and $Q(\omega,t)$ is optimal in the sense that
\begin{align}
& E(H_\Gamma(\mathcal F)(w,t);Q(\omega,t)) := \sum_{i=1}^m \|H_\Gamma f_i(\omega,t) - \Proj_{Q(\omega,t)}H_\Gamma(f_i)(w,t)\|^2_{\ell^2(\B^\perp)} \nonumber \\
& \leq \sum_{i=1}^m \|H_\Gamma f_i(\omega,t) - \Proj_{Q'}H_\Gamma(\mathcal F)(w,t)\|^2_{\ell^2(\B^\perp)} := E(H_\Gamma(f_i)(w,t);Q')  \label{Eq3-3}
\end{align}
for any $Q'$ subspace of $\ell^2(\B^\perp)$ of dimension less than or equal to $n$. Moreover,
\begin{equation}\label{eq:pointwiseerror}
E(H_\Gamma(\mathcal F)(w,t);Q(\omega,t)) = \sum_{i = n+1}^m \lambda_i(\omega,t).
\end{equation}

Before continuing with the proof, let us relate the pointwise errors that appear in \eqref{Eq3-3} to the error defined in \eqref{Eq1-1} for $\Gamma$-invariant subspaces. 

\begin{proposition} \label{Pro2}
	For $V = S_\Gamma(\mathcal A)$ as in Proposition \ref{Pro1},
	$$
	\mathcal E(\mathcal F; V) = \int_{T_{L^\perp}} \int_{T_{\B^\perp}} E(H_\Gamma(\mathcal F)(w,t);J_V(\omega,t)) \, dt d\omega\,,
	$$
	where $J_V(\omega, t)$ is defined in \eqref{EqJ}.
\end{proposition}

\begin{proof}
By $2)$ of Theorem \ref{Th2-1}, $H_\Gamma$ is an isometry from $L^2(G)$ onto the space $L^2(T_{L^\perp}\times T_{\B^\perp}, \ell^2(\B^\perp)).$ Therefore,
\begin{align*}
\mathcal E (\mathcal F; V) &= \sum_{j =1}^m \| f_j - \proj_V f_j\|_{L^2(G)}^2 \\ &= \sum_{j =1}^m \| H_\Gamma f_j - H_\Gamma \proj_V f_j\|_{L^2(T_{L^\perp}\times T_{\B^\perp}, \ell^2(\B^\perp))}^2  \\
& = \sum_{j =1}^m \int_{T_{L^\perp}} \int_{\B_{L^\perp}} \| H_\Gamma f_j (\omega,t) - H_\Gamma \proj_V f_j (\omega,t)\|_{ \ell^2(\B^\perp)}^2\, \,dt d\omega \,.
\end{align*}
By Proposition \ref{Pro1},
\begin{align*}
\mathcal E (\mathcal F; V) & =  \int_{T_{L^\perp}} \int_{\B_{L^\perp}} \sum_{j =1}^m \| H_\Gamma f_j (\omega,t) - \Proj_{J_V(\omega,t)}(H_\Gamma f_j(\omega, t))\|_{ \ell^2(\B^\perp)}^2 \, dt d\omega \\
& = \int_{T_{L^\perp}} \int_{\B_{L^\perp}}E(H_\Gamma(\mathcal F)(w,t);J_V(\omega,t)) \, dt d\omega\,. & \hfill \Box
\end{align*}
\end{proof}

Let us now continue with the proof of Theorem \ref{Thm:main}. By definition \eqref{Eq3-2}, each $q_i(\omega,t)$ is  measurable and defined on $T_{L^\perp}\times T_{\B^\perp}$ with values in  $\ell^2(\B^\perp)$. Moreover,
\begin{align*}
   &\|q_i(\omega,t)\|^2_{\ell^2(\B^\perp)} = \langle q_i(\omega,t), q_i(\omega,t) \rangle_{\ell^2(\B^\perp)} \\
   & = \widetilde \sigma_i(\omega,t)^2 \sum_{b\in \B^\perp} \sum_{j=1}^m \sum_{s=1}^m y_{i,j}(\omega, t)\, \mathcal Z f_j(\omega,t+b)\, \mathcal Z f_s(\omega,t+b)\, \overline{ y_{i,s}(\omega, t)}\\
   & = \widetilde \sigma_i(\omega,t)^2 \sum_{j=1}^m  y_{i,j}(\omega, t) \sum_{s=1}^m \langle Z f_j(\omega,t), Z f_s(\omega,t) \rangle_{\ell^2(\B^\perp)} \overline{ y_{i,s}(\omega, t)}\,.
\end{align*}
In matrix form,
$$
\|q_i(\omega,t)\|^2_{\ell^2(\B^\perp)} = \widetilde \sigma_i(\omega,t)^2 y_i(\omega,t) \, G_{\mathcal F, \Gamma}(w,t)\, \overline {y_i(\omega,t)}^t\,. 
$$
By \eqref{Eq3-1}, the orthonormality of the vectors $y_i(\omega, t)$, and the definition of $\widetilde \sigma_i(\omega,t)$, we have
$$
\|q_i(\omega,t)\|^2_{\ell^2(\B^\perp)} = \widetilde \sigma_i(\omega,t)^2 \lambda_i (\omega, t) \, \|y_i(\omega,t)\|^2_{\textbf{}} \leq 1\,.
$$
Since $T_{L^\perp}$ and $T_{\B^\perp}$ have finite measure, we conclude that for $i=1, \dots, n,$ $q_i \in L^2(T_{L^\perp}\times T_{\B^\perp}, \ell^2(\B^\perp))$. The mapping $H_\Gamma$ is onto by part $2)$ of Theorem \ref{Th2-1}. Therefore there exist $\psi_i\in L^2(G)$ such that 
$$
H_\Gamma(\psi_i) = q_i\,, \quad  i= 1, \dots, n.
$$

It remains to show that the space $W:= S_\Gamma({\psi_1, \dots, \psi_n})$ is the optimal one as required in the statement of Theorem \ref{Thm:main}. 

By Proposition \ref{Pro2}
$$
\mathcal E(\mathcal F; W) = \int_{T_{L^\perp}} \int_{T_{\B^\perp}} E(H_\Gamma(\mathcal F)(w,t);J_W(\omega,t)) \, dt d\omega\,.
$$
By \eqref{Eq3-3} and the definitions of $\psi_i$, $J_W(\omega,t) = Q(\omega,t)$. Therefore, we can write,
\begin{equation}\label{eq:integralerror}
\mathcal E(\mathcal F; W) = \int_{T_{L^\perp}} \int_{T_{\B^\perp}} E(H_\Gamma(\mathcal F)(w,t);Q(\omega,t)) \, dt d\omega\,.
\end{equation}

Let now $V=S_\Gamma({\varphi_1, \dots, \varphi_r})$, $r\leq n$, be any $\Gamma$-invariant subspace of length less than or equal $n$. Since $J_V(\omega,t)$ has dimension less than or equal $n$, \eqref{Eq3-3} gives
$$
\mathcal E(\mathcal F; W) \leq \int_{T_{L^\perp}} \int_{T_{\B^\perp}} E(H_\Gamma(\mathcal F)(w,t);J_V (\omega,t)) \, dt d\omega = \mathcal E(\mathcal F; V) \,,
$$
where the last equality is due to Proposition \ref{Pro2}. Moreover, by (\ref{eq:integralerror}) and (\ref{eq:pointwiseerror})
$$
\mathcal E(\mathcal F; W) = \sum_{i=n+1}^m \int_{T_{L^\perp}} \int_{T_{\B^\perp}} \lambda_i(\omega,t) d\omega dt.
$$
This finishes the proof of Theorem \ref{Thm:main}. \hfill $\Box$

\section{Appendix} \label{Appendix}

We give the proof of the following Lemma that has been used in Section \ref{Sec2} to prove part $2)$ of Theorem \ref{Th2-1}. 

\begin{lemma} \label{Lemma}
Let $\sigma : \mathbb H_1 \longrightarrow \mathbb H_2$ be an isometric isomorphism between the Hilbert spaces  $\mathbb H_1$ and $\mathbb H_2$. For a measure spaces $(X, d\mu)$ the map $Q_\sigma : L^2(X, \mathbb H_1) \longrightarrow L^2(X, \mathbb H_2)$ given by $(Q_\sigma f)(x) = \sigma(f(x))$ is also an isometric isomorphism.	
\end{lemma}

\begin{proof}
Let $f$ be a measurable vector function in $L^2(X, \mathbb H_1)$, that is, for every $y\in \mathbb H_1$ the scalar function $x \longrightarrow \langle f(x) , y \rangle_{\mathbb H_1}$ is measurable. We must prove that $Q f$ is also a measurable vector function in $L^2(X, \mathbb H_2)$. For $z\in \mathbb H_2$ we have
$$
<Qf(x),z>_{\mathbb H_2}= <\sigma(f(x)),z>_{ \mathbb H_2}=<f(x),\sigma^*(z)>_{\mathbb H_1}.
$$ 
Since $\sigma^*(z) = \sigma^{-1}(z)$ is a general element of $\mathbb H_1$, this shows that $Q f$ is measurable.
Moreover, for $f,g \in L^2(X, \mathbb H_1)$,
\begin{align*}
<Qf,Qg>_{L^2(X, \mathbb H_2)} =&\int_X<\sigma(f(x)),\sigma(g(x))>_{\mathbb H_2} d\mu(x)\\
= &\int_X<f(x),(g(x)>_{\mathbb H_1} d\mu(x)=<f,g>_{L^2(X, \mathbb H_1)}\,.
\end{align*}
This shows that if $f\in L^2(X, \mathbb H_1)$, $Q_\sigma f \in L^2(X, \mathbb H_2)$ and that $ Q_\sigma$ is and isometry.

Finally, it is easy to see that $R: L^2(X, \mathbb H_2) \rightarrow L^2(X, \mathbb H_1)$ defined by $Rg(x) =\sigma^{-1}(g(x)$ is the inverse and the adjoint of $Q$. Therefore, $Q_\sigma$ is onto.
\end{proof}

\section{Acknowledgements}
This project has received funding from the European Union's Horizon 2020 research and innovation programme under the Marie Sk\l{}odowska-Curie grant agreement No 777822. 

In addition, D. Barbieri and E. Hern\'andez were supported by Grant MTM2016-76566-P (Ministerio de Ciencia, Innovaci\'on y Universidades, Spain). C. Cabrelli and U. Molter were supported by Grants UBACyT 20020170100430BA (University of Buenos Aires), PIP11220150100355 (CONICET) and PICT 2014-1480 (Ministerio de Ciencia, Tecnolog\'{i}a e Innovaci\'on, Argentina).

\end{document}